\documentclass[pdflatex,sn-mathphys-num]{sn-jnl}

\usepackage{graphicx}%
\usepackage{multirow}%
\usepackage{amsmath,amssymb,amsfonts}%
\usepackage{amsthm}%
\usepackage{mathrsfs}%
\usepackage[title]{appendix}%
\usepackage{xcolor}%
\usepackage{textcomp}%
\usepackage{manyfoot}%
\usepackage{booktabs}%
\usepackage{algorithm}%
\usepackage{algorithmicx}%
\usepackage{algpseudocode}%
\usepackage{listings}%

\theoremstyle{thmstyleone}%
\newtheorem{theorem}{Theorem}
\theoremstyle{thmstyletwo}%
\newtheorem{lemma}{Lemma}%
\newtheorem{corollary}{Corollary}%
\theoremstyle{thmstylethree}%

\begin{document}

\title[Cardinal functions and mappings associated with the space of quasi-continuous functions equipped with topology of point-wise convergence]{Cardinal functions and mappings associated with the space of quasi-continuous functions equipped with topology of point-wise convergence}

\author*[1]{\fnm{Sanjay} \sur{Mishra}} \email{drsmishraresearch@gmail.com}

\author[2]{\fnm{Chander Mohan} \sur{Bishnoi}} \email{chandermohan.cm.b@gmail.com}
\equalcont{This author contributed equally to this work.}


\affil*[1]{\orgdiv{Department of Mathematics,
Amity School of Applied Sciences}, \orgname{Amity University}, \orgaddress{\street{ Lucknow Campus}, \city{Lucknow}, \postcode{226028}, \state{Uttar Pradesh}, \country{India}}}

\affil[2]{\orgdiv{Department of Mathematics}, \orgname{Lovely Professional University}, \orgaddress{\street{Jalandhar - Delhi G.T. Road}, \city{Phagwara}, \postcode{144411}, \state{Punjab}, \country{India}}}



\abstract{Cardinal functions provide valuable insight into the topological properties of spaces, helping to analyze and compare spaces in terms of their covering, convergence and separation properties. This paper focuses on investigating cardinal functions like network weight, Lindelöf degree, tightness, weak covering, pseudocharacter, and $i$-weight, for the spaces $Q_{P}(X)$ and $Q_{P}(X,Y)$ of quasi-continuous functions under the topology of point-wise convergence. In addition to these, we also investigate properties of restriction and induced maps associated with the spaces $Q_{P}(X)$ and $Q_{P}(X,Y)$.}





\maketitle

\section{Introduction}\label{sec1}
Before we delve deeper into the content of this paper, it is necessary for us to initially explore the important symbols and explanations that will be used. The following notation to be employed includes $X, Y$, and $Z$, which are to be treated as topological spaces unless explicitly indicated otherwise. The spaces $C_{P}(X)$ and $C_{P}(\mathbb{R})$ represent collections of continuous real-valued functions defined on the space $X$ and the set of real numbers $\mathbb{R}$, respectively, each equipped with the topology of point-wise convergence. Similarly, the spaces $Q_{P}(X)$ and $Q_{P}(\mathbb{R})$ represent collections of quasi-continuous real-valued functions defined on the space $X$ and the set of real numbers $\mathbb{R}$, respectively, each equipped with the topology of point-wise convergence. The spaces $F_{P}(X,Y)$ and $Q_{P}(X,Y)$ are collections of functions and quasi-continuous functions from space $X$ to space $Y$, respectively, each equipped with a topology of point-wise convergence. The first infinite (countable) cardinal no., uncountable (second uncountable) cardinal no., and arbitrary cardinal number are denoted by $\aleph_{0}$, $\aleph_{0}$, and $\eta$, respectively. The closure, interior, and complement of any set $A$ are denoted by $\overline{A}$, $A^{\circ}$, and $A^{C}$, respectively. The collection of all finite subsets of $X$ is denoted by $\mathcal{F}$.\\In 1899, Baire first found the condition for a function to be a quasi-continuous function in the paper \cite{baire1899}. During the study of continuity points of separately continuous functions from ${\mathbb{R}}^{2}$ to $\mathbb{R}$, he determined that a function of two variables is quasi-continuous if it is continuous in each variable separately. Later, the quasi-continuity introduced in the paper \cite{Kempisty1932} for real functions of several real variables was thoroughly and extensively tested by Kempisty in 1932. When we take a closer look, it turns out that the researchers found this study interesting for a few reasons. The two most important of these are: The first reason is that it provides a relatively good connection to continuity while being a more general concept. This relationship allows mathematicians to extend concepts and theorems from continuous functions to a broader class of functions. The second reason is that it has deep connections to mathematical analysis and topology. Understanding these functions helps in exploring and solving problems in these areas more comprehensively. The reader can see survey articles authored by Piotrowski, specifically \cite{Piotrowski1985} and \cite{Piotrowski1987} that encompass a range of intriguing outcomes in this area. However, these papers do not exclusively focus on this subject. In addition to these papers, readers are encouraged to see the survey paper \cite{Neubrunn1988} written by Neubrunn, which is a collection of interesting and important results on quasi-continuous functions. Quasi-continuous functions played an important role in the study of topological groups, the characterization of minimal usco, minimal cusco maps, and the CHART group, which is the key object for the study of topological dynamics. Additionally, recent research by Bishnoi and Mishra \cite{Bishnoi2023} in 2023 has demonstrated how strong forms of connectedness are preserved through quasi-continuous functions. For a more in-depth exploration of these applications, one can refer to \cite[Chapter 2]{Hola2021}.
Next, we aim to shed some light on the concise history of topological function spaces. The introduction of the theory of function spaces equipped with the topology of point convergence, now called $C_{P}$-theory, is attributed to Alexander Vladimirovich Arhangel’skii. In 1992, Arhangel’skii studied important results of $C_{P}$-theory in \cite{Arkhangelskii992}. Subsequently, numerous researchers dedicated their efforts to enhancing $C_{P}$-theory, bestowing upon it the elegance and magnificence it possesses today. A huge collection of results on $C_{P}$-theory along with many open problems provided by Tkachuk in \cite{Tkachuk2010, Tkachuk2014, Tkachuk2015, Tkachuk2016}. In 2018, McCoy et al. \cite{McCoy2018} collected general results on topological function spaces with different topologies such as uniform topology, fine topology, and graph topology in \cite{McCoy2018}. More recently, in 2023, Mishra and Bhaumik \cite{Mishra2023} and Aaliya and Mishra \cite{Aaliya2023} studied properties of topological function spaces under Cauchy convergence topology and regular topology, respectively.
Now let us see how the study of topological function space begins, where functions were considered as quasi-continuous functions. Hola and Holy studied various properties of the space of quasi-continuous functions under different topologies in the literature \cite{Hola2016, Hola2017, Hola2018}. Recently, in 2020, Hola and Holy studied cardinal invariants of the space $Q_{P}(X)$. Furthermore, in 2022 Kumar and Tyagi studied cardinal invariants of the space $Q_{P}(X,Y)$ in \cite{Kumar2022} a more general form. Extending these investigations, we present in this paper additional findings concerning cardinal invariants of the space $Q_{P}(X)$ and further study some maps on both $Q_{P}(X)$ and $Q_{P}(X,Y)$ spaces.
The structure of this paper is as follows: In the section \eqref{s:Preliminaries}, we recall some important definitions and results that help to understand our main results. In section \eqref{s:Cardinal functions for the spaces qpx and qpxy} we study cardinal functions like weight, network weight, Lindelöf degree, tightness, weak covering, pseudocharacter, and $i$-weight for the spaces $Q_{P}(X)$ and $Q_{P}(X,Y)$ under a certain condition on the space $X$. Further at the end in section \eqref{s:Mappings on the spaces qpx and qpxy}, we studied the restriction map and induced map for the spaces $Q_{P}(X)$ and $Q_{P}(X,Y)$. \section{Preliminaries}\label{s:Preliminaries}According to Neubrunn in \cite{Neubrunn1988}, a function $f \colon X \to Y$ is said to be quasi-continuous at $x \in X$ if for any open sets $U$ and $V$ such that $x \in U, f(x) \in V$ there exists a non-empty open subset $G$ of $U$ such that $f(G) \subset V$. The function $f$ is called quasi-continuous on $X$ if $f$ is quasi-continuous for all $x \in X$. It is easy to see that every continuous function is quasi-continuous, but the converse need not be true. For example, any only monotone left or right continuous function $f \colon \mathbb{R} \to \mathbb{R}$ is quasi-continuous but not continuous.

Now we are going to recall some definitions related to the topology of the function spaces.
\begin{enumerate}
\item The collection
\begin{gather*}
\mathcal{S} = \{S(x,U) \colon  x \in X, U \, \text{open in} \, Y\}, \,\text{where} \\
S(x,U) = \{f \in F(X,Y) \colon f(x) \in U \}
\end{gather*}
is subbase for the topology on the set $F(X,Y)$, called the topology of point-wise convergence.
\item The collections
\begin{gather*}
\mathcal{S}' = \{[x,U] \colon  x \in X, U \, \text{open in} \, Y\}, \,\text{where} \\
[x, U] = \{f \in Q(X,Y) \colon f(x) \in U \}
\end{gather*}
and
\begin{gather*}
\mathcal{B} =\{[x_{1}, \dots, x_{n}; U_{1}, \dots, U_{n}]\colon x_{i} \in X, U_{i} \, \text{open in} \, Y\}, \text{where} \\
 [x_{1}, \dots, x_{n}; U_{1}, \dots, U_{n}] = \{f \in Q(X,Y) \colon f(x_{i}) \in U_{i}, 1 \leq i \leq n\}
\end{gather*}
are respectively subbase and base for a topology on the set $Q(X,Y)$ which is called the topology of point-wise convergence. The space $Q_{P}(X,Y)$ is the subspace of $F_{P}(X,Y)$.
\item For metric space $Y$ and $\mathcal{F}$ be collection of all finite subset of $X$. Thus, for $\epsilon>0$ the set
\[W(f,A,\epsilon) = \{g \in Q(X,Y) \colon d(f(x), g(x)) < \epsilon, \forall x \in A \in \mathcal{F} \}\] is an open neighbourhood of $f$ in $Q_{P}(X,Y)$.
\end{enumerate}

In the study of topology, cardinal functions are numerical characteristics that measure various aspects of the size, covering, and convergence properties of topological spaces. It helps to understand and compare the topological properties of different spaces based on their cardinalities. In this paper, our main objective is to discuss the cardinal functions for the spaces $Q_{P}(X)$ and $Q_{P}(X,Y)$ under certain conditions on the space $X$. For better understanding, let us first recall the cardinal functions for the space $X$. If $X$ is topological space, then
\begin{enumerate}
\item the weight of the space $X$ is
\begin{equation}\label{eq:Weight of X}
w(X) = \aleph_{0} + \min\{|\mathcal{B}| \colon \mathcal{B}\, \text{is a basis in} \, X\}
\end{equation}
\item the network weight of the space $X$ is
\begin{equation}\label{eq:Newtwork weight of X}
nw(X) = \aleph_{0} + \min \{|\mathcal{N}|\colon \mathcal{N} \,\text{is a network of} \, X\}
\end{equation}
where the network of a space $X$ is a collection $\mathcal{N}$ of subsets of $X$ such that for any $ x \in X$ and every open set $U$ containing $x$, there exists $N \in \mathcal{N}$ such that $ x \in N \subset U$.
\item the $i$-weight of the space $X$ is
\begin{equation}\label{eq:i-weight of X}
iw(X) = \aleph_{0} + \min\{w(Y) \colon \exists \, \text{continuous and bijective} \, f \colon X \to Y\}
\end{equation}
\item the weak covering number of the space $X$ is
\begin{equation}\label{eq:Weak covering number of X}
wc(X) = \aleph_{0} + \min\{|\mathcal{J}|\colon \mathcal{J}\, \text{is a weak covering of} \, X\}
\end{equation}
where a weak covering of a space $X$ is a collection $\mathcal{J}$ of open subsets of $X$ such that $\overline{\bigcup \mathcal{J}} = X$.
\item the tightness of the space $X$ is
\begin{gather}\label{eq:Tightness of X}
 t(X) = \{t(x,X) \colon x \in X\}, \, \text{where}  \\
 t(x,X) = \aleph_{0} + \sup\{a(x, Y) \colon x \in \overline{Y} \subset X\}
\end{gather}
is tightness at the point $x \in X$  and $a(x, Y) = \min\{|Z| \colon Z \subset Y, x \in \overline{Z}\}$.
\item the density of the space $X$ is
\begin{equation}\label{eq:Density of X}
d(X) = \aleph_{0} + \min\{|D|\colon D \, \text{is a dense set in} \, X\}
\end{equation}
\item the character of the space $X$ is
\begin{gather*}\label{eq:character of X}
\chi(X) = \sup\{\chi(x, X) \colon x \in X\}, \, \text{where}   \\
\chi(x, X) = \aleph_{0} + \min \{|\mathcal{O}| \colon \mathcal{O} \, \text{is local base at} \, x\}
\end{gather*}
is character of the point $x \in X$.
\item the pseudocharacter of the space $X$ is
\begin{gather}\label{eq:Pseudocharacter of X}
\psi(X) = \sup\{\psi(x, X) \colon x \in X\}, \, \text{where} \\
\psi(x, X) = \aleph_{0} + \min \{|\gamma| \colon \gamma \, \text{is a family of open sets in} \, X \, \text{such that} \, \cap \gamma = \{x\}\}
\end{gather}
is pseudocharacter of the point $x \in X$. The pseudocharacter of the subset $A$ of the space $X$ is
\begin{equation}\label{Pseudocharacter of subset A of X}
\psi(A, X) = \min \{|\mathcal{U}| \colon \mathcal{U} \subset  T(X), \bigcup \mathcal{U} = A\}
\end{equation}
\end{enumerate}
We have an interrelation between cardinal functions.
\begin{gather}\label{eq:Interrelation between cardinal functions}
d(X) \leq nw(X) \leq w(X),\, \text{where}\, X \, \text{is arbitrary space} \\
d(X) = nw(X) = w(X), \, \text{where}\, X \, \text{is metrizable space}
\end{gather}
For a more comprehensive understanding of cardinal functions associated with topological spaces, refer to \cite{Kunen2014}.

Let's look at some recent interesting results for cardinal invariants of the space $Q_{P}(X)$ and $Q_{P}(X,Y)$. In 2020, Hola and Holy established the following results in \cite{Hola2020}:
\begin{enumerate}
\item For space $X$,
\begin{gather}\label{result1}
\chi(Q_{P}(X)) = \pi_{\chi}(Q_{P}(X)) = |X| \\
d(Q_{P}(X)) \leq w(X) \\
w(Q_{P}(X)) = |X|
\end{gather}
\item For Tychonoff space $X$,
\begin{gather}\label{result2}
d(Q_{P}(X)) = iw(X)
\end{gather}
\item For regular space $X$,
\begin{gather}\label{result3}
nw(X) \leq nw(Q_{P}(X)) \leq |X|
\end{gather}
\item For Tychonoff space $X$ \cite[Theorem 4.1.2]{McCoy1988},
\begin{gather}\label{result3}
nw(X) = nw(C_{P}(X)) \leq   nw(Q_{P}(X))
\end{gather}
\item From \cite[Example 4.1.4]{Hola2020},
\begin{equation}\label{eq:1}
nw(\mathbb{R}) = w(\mathbb{R}) = \aleph_{0} \neq nw(Q_{P}(\mathbb{R})) = |\mathbb{R}| = \aleph_{1}
\end{equation}
Therefore, for arbitrary space $X$,
\begin{equation}\label{eq:2}
nw(X) \neq nw(Q_{P}(X))
\end{equation}
\end{enumerate}

Based on the ongoing research initiated by Hola and Holy in \cite{Hola2020}, Kumar and Tyagi extended the investigation in 2022 by describing the following cardinal invariant of the space $Q_{P}(X,Y)$ with a broader perspective in \cite{Kumar2022}.

\begin{enumerate}
\item For Hausdorff space $X$ and non-trivial $T_{1}$-space $Y$,
\begin{equation}\label{eq:Cardinal invariants of spaces qpxy 1}
|X| \leq \pi_{\chi}(Q_{P}(X,Y)) \leq \chi(Q_{P}(X,Y)) \leq \omega(Q_{P}(X,Y))
\end{equation}
\item For infinite Hausdorff space $X$ and non-trivial second countable $T_{1}$-space $Y$,
\begin{equation}\label{eq:Cardinal invariants of spaces qpxy 2}
|X| = \pi_{\chi}(Q_{P}(X,Y)) = \chi(Q_{P}(X,Y)) = \omega(Q_{P}(X,Y))
\end{equation}
\item For Hausdorff space $X$,
\begin{gather}\label{eq:Cardinal invariants of spaces qpxy 3}
d(Q_{P}(X,Y)) \leq \omega(X) \cdot d(Y) \text{and} \\
Q_{P}(X,Y) \, \text{is dense in} \, F_{P}(X,Y)
\end{gather}
\item For second countable Hausdorff space $X$ and separable space $Y$,
\begin{equation}\label{eq:Cardinal invariants of spaces qpxy 3}
Q_{P}(X,Y) \, \text{is separable}
\end{equation}
\item For regular space $X$ and non-trivial space $Y$,
\begin{equation}\label{eq:Cardinal invariants of spaces qpxy 4}
d(X) \leq \psi (Q_{P}(X,Y))
\end{equation}
\item For Tychonoff space $X$ and separable space $Y$, \cite[Problem 034]{Tkachuk2010}
\begin{equation}\label{eq:Cardinal invariants of spaces qpxy}
C_{P}(\mathbb{R}) \, \text{is dense in} \, Q_{P}(\mathbb{R})
\end{equation}
\item For Hausdorff space $X$, separable space $Y$ and $\mathcal{F}$ is locally finite family of nonempty open subsets of $Q_{P}(X,Y)$, $\mathcal{F}$ is countable.
\end{enumerate}

For more details on cardinal invariants for the space of continuous functions, see \cite{Arkhangelskii992, Tkachuk2010}.

Let us look at some important lemmas which help us to prove the main results in the next section.
\begin{lemma}\label{l:Condition for quasicont}\cite[Lemma 4.2]{Hola2020}
If $X$ and $Y$ are spaces and $f \colon X \to Y$ is a function with property: for any $x \in X$, there exists an open subset $U$ of $X$ such that $x \in \overline{U}$ and $f(y) = f(x)$ for any $y \in U$. Then $f$ is quasi-continuous.
\end{lemma}

\begin{lemma}\label{l:Quasicontinuous function}\cite[Lemma 4.12]{Kumar2022}
If $X$ is a regular space and $Y$ is any space, then for any $x \in X$, non-empty closed set $E \subseteq X$ such that $x \notin E$ and $y_{1}, y_{2} \in Y$, then there exists $f \in Q_{P}(X,Y)$ such that  $f(x) = y_{1}$ and $f(E) = \{y_{2}\}$.
\end{lemma}

\begin{lemma}\label{l:Evalution map}\cite[Lemma 5.5]{Kumar2022}
The evaluation map $e_{x}\colon Q_{P}(X,Y) \to Y$ at any $x \in X$, defined by $e_{x}(f) = f(x)$ is continuous.
\end{lemma}
\section{Cardinal functions for the spaces $Q_{P}(X)$ and $Q_{P}(X,Y)$}\label{s:Cardinal functions for the spaces qpx and qpxy}
As we know for an arbitrary space $X$, $nw(X) \leq w(X)$ but for matrizable space $X$ network weight and weight coincide. In addition to the metrizable space $X$, the problem of the coincidence of network weight and weight can be solved in terms of the function space. However, according to Hola and Holy \cite[Example 5.1]{Hola2020}, $nw(C_{P}(\mathbb{R})) \neq w(C_{P}(\mathbb{R}))$, since $nw(C_{P}(\mathbb{R})) = \aleph_{0}$ and $w(C_{P}(\mathbb{R})) = \aleph_{1}$. In continuation, the result  \cite[Corollary 4.11]{Hola2020} tells us the coincidence of network weight and weight for the space $Q_{P}(X)$ for countable space $X$ but our next result is a coincidence of network weight and weight for the space $Q_{P}(X)$ in the more general form of the space $X$.

\begin{theorem}
For a ordered Hausdorff space $X$, $nw(Q_{P}(X)) = w(Q_{P}(X)) = |X|$.
\end{theorem}

\begin{proof}
Let $X$ is ordered Hausdorff space. Then, for each $r\in X$, we choose an open set $U_r$ such that $x\in U_{r}$ for all $x<r$.  We define a  function $f_{r} \colon X \to \mathbb{R}$  as,
\[
f_{r}(x) =
\begin{cases}
1, & \text{if $x \in \overline{U_{r}} $};\\
0, & \text{otherwise}.
\end{cases}
\]
By Lemma \eqref{l:Condition for quasicont} is quasi-continuous.
Let $\mathcal{N}$ is network for the space $Q_{P}(X)$ such that $|\mathcal{N}| = nw(Q_{P}(X))$. Fix $1>\epsilon>0$, then  $\{W(f_{r}, \{r\}, \epsilon) \colon r \in X,\}$ is collection of open neighborhoods of $f_{r} \in Q_{P}(X)$. By the definition of network, for all $f_{r} \in Q_{P}(X)$ there exists $N_{r} \in \mathcal{N}$ such that $f_{r} \in N_{r} \subset W(f_{r}, \{r\}, \epsilon)$.
The map $g \colon X \to \mathcal{N}$ is defined as $g(r) = N_{r}$ for any $r\in X$. Now we claim $|X| \leq |\mathcal{N}|$. To prove our claim, we show that the map $g \colon X \to \mathcal{N}$ defined as $g(r) = N_{r}$ for any $r\in X$, is injective. Let us choose two distinct elements $r_{1}$ and $r_{2}$ of $X$. By the definition of the function $f_{r}$, $f_{r_{1}} \notin W(f_{r_{2}}, \{r_{2}\}, \epsilon)$ i.e. $N_{r_{1}} \nsubseteq  W(f_{r_{2}}, \{r_{2}\}, \epsilon)$. Since $N_{r_{2}} \subseteq  W(f_{r_{2}}, \{r_{2}\}, \epsilon)$ so $N_{r_{1}} \neq N_{r_{2}}$. Since $|\mathcal{N}| = nw(Q_{P}(X))$ and $|X| \leq |\mathcal{N}|$, therefore $|X| \leq nw(Q_{P}(X))$. By result \cite[Theorem 4.9]{Hola2020}, $w(Q_{P}(X)) = |X|$, therefore $w(Q_{P}(X)) \leq nw(Q_{P}(X))$. But by the result (2.13), $nw(Q_{P}(X)) \leq w(Q_{P}(X))$, therefore $nw(Q_{P}(X)) = w(Q_{P}(X))$.
\end{proof}

We are now going to investigate the cardinal function for the space $Q_{P}(X)$ in terms of the covering property. Let us define some important definitions and results that help in proving the next result. The Lindelöf space is defined by Alexandroff and Urysohn in 1929 as a space $X$ is called Lindelöf or has Lindelöf property if every open cover of $X$ is reducible to a countable subcover. The Lindelöf degree of a space $X$ is a cardinal number that provides a measure of the ``smallness" of the space in terms of its covering properties. Symbolically the Lindelöf degree of a space $X$ is defined by $L(X) = \aleph_{0} + \inf\{\eta \colon \text{any open cover } \mathcal{V} \, \text{of}\,  X \,\text{has a subcover}\, \mathcal{U} \subseteq \mathcal{V}\, \text{and} \, |\mathcal{U}| \leq \eta\}$. When investigating topological function spaces, it is observed that the tightness of such spaces is intricately linked to the Lindelöf degree of the underlying base space. This relation becomes clear in the ensuing result \cite[Theorem 4.7.1]{McCoy1988} $\alpha L(X) = t(C_{\alpha}X)$, where $\alpha$ is collection of subsets of $X$. Motivated by this result, we set out to explore the relation between the tightness of the space $Q_{P}(X)$ and the Lindelöf degree of the space $X$ in our next result.

\begin{theorem}\label{t:Lindelof degree and tightness of qpx}
For regular space $X$, then $L(X)\leq t(Q_{P}(X))$.
\end{theorem}
\begin{proof}
Let $ t(Q_{P}(X))= \eta$ and $\mathcal{O}$ be any open cover of $X$. Let us choose a subcollection $\mathcal{F}' = \{A \in \mathcal{F} \colon A \subset O_{A}, O_{A} \in \mathcal{O}\}$  of $\mathcal{F}$. By the Lemma \eqref{l:Quasicontinuous function}, for each $A \in \mathcal{F}'$ there exists $f_{A} \in Q_{P}(X)$ such that $f_{A}(O_{A}) = \{0\}$ and $f_{A}(X \backslash O_{A}) = \{1\}$. Let us construct a subset $P = \{f_{A} \colon A \in \mathcal{F}'\}$ of $Q_{P}(X)$. By the definition of $f_{A}$, it is clear that zero function $f_{0} \in \overline{P}$. By the definition of tightness of the space $Q_{P}(X)$, there exists a subset $P'$ of $P$ such that $|P'| \leq \eta$ and $f_{0} \in \overline{P'}$. Let us choose a subcollection $\mathcal{O}' = \{O_{A} \colon f_{A} \in P'\}$ of $\mathcal{O}$, where $|\mathcal{O}'| \leq \eta$. Now we claim $\mathcal{O}'$ is cover of $X$. For this, we show that for all $x \in X$, there exists $O_{A} \in \mathcal{O}'$ such that $x \in O_{A}$. Let us consider an open neighborhood $[x, (-1, 1)] = \{g \in Q_{P}(X) \colon g(x) \in (-1, 1)\}$ of the zero function $f_{0}$. Since $f_{0} \in \overline{P'}$ so there exists some $f_{A} \in P'$ such that $f_{A} \in P' \cap [x, (-1, 1)]$. Thus $f_{A}(x) < 1$, but by definition of $f_{A}$, $f_{A}(x) = 0$. Therefore, $x \in O_{A}$. Finally by definition of Lindelof degree, $L(X)\leq t(Q_{P}(X))$.
\end{proof}
\begin{corollary}
For regular space $X$, if $t(Q_{P}(X)) = \aleph_{0}$, then $X$ is Lindelof space.
\end{corollary}
According to the Example \cite[Example 5.1]{Hola2020}, $\psi(Q_{P}(\mathbb{R})) = \aleph_{1}$. But we know that $wc(\mathbb{R}) = \aleph_{0}$ i.e.  $\psi(Q_{P}(\mathbb{R}))\neq wc(\mathbb{R})$. An emerging question is how the pseudocharacter and weak covering numbers are generally interconnected. Our next result tells us that under the condition of regularity of space $X$ and for any metric space $Y$, $wc(X)\leq \psi(Q_{P}(X,Y))$.

\begin{theorem}\label{t:weak covering pseudo character}
For regular space $X$ and $Y$ be metric space, then $wc(X)\leq \psi(Q_{P}(X,Y))$.
\end{theorem}

\begin{proof}
Let $h_{0} \in Q_{P}(X,Y)$ such that $h_{0}(x) = b$, for all $x \in X$, where $b$ is fixed in $Y$. Let $\psi (h_{0}, Q_{P}(X,Y)) = |\mathcal{V}|$, where $\mathcal{V} = \{W(h_{0}, A_{i}, \epsilon)\colon i \in J\}$ is collection of open neighborhoods of $h_{0} \in Q_{P}(X,Y)$, such that $\bigcap \mathcal{V} = \{h_{0}\}$, $A_{i} \in \mathcal{F}$ and $J$ is arbitrary index set.
Let us construct a collection $\mathcal{B}_{A_{i}} = \{U_{x} \colon x \in A_{i}\}$ of open sets of $X$ such that for each $x \in A_{i}$, we can pick an open set $U_{x}$ that contain $x$. Let $\mathcal{J} = \bigcup\{\mathcal{B}_{A_{i}}\colon i \in J\}$, where $|\mathcal{J}| \leq |J|$.
We claim that $\mathcal{J}$ is the weak covering of $X$. Let $x \in X$ and $x \notin \overline{\bigcup \mathcal{J}}$. Then by Lemma \eqref{l:Quasicontinuous function} there exists a map $h \in Q_{P}(X,Y)$ such that $h(x) = a$ and $h(\overline{\bigcup \mathcal{J}})=\{b\}$ where $a \neq b$. But each $W(h_{0},A_{i},\epsilon)$ contains $h$, which is a contradiction. Thus, $\overline{\bigcup \mathcal{J}} = X$. Since $\mathcal{J}$ is a weak cover of $X$, therefore, $wc(X) \leq \psi(Q_{P}(X,Y))$.
\end{proof}

By the Problem \cite[Problem 175]{Tkachuk2010}, the separability of the Tychnoff space $X$ is related to the pseudocharacter of some compact subspace $G$ of $C_{P}(X)$ that is $X$ is separable if $\psi(G, Q_{P}(X,Y))\leq \aleph_{0}$. Now, in the next result, we are going to find the condition for the separability of regular space $X$ with the help of the pseudocharacter of some compact subspace of the $Q_{P}(X,Y)$, where $Y$ is ordered topological space.

\begin{theorem}
Let $X$ is  regular space, $Y$ ordered set with ordered topology, if $\psi(G, Q_{P}(X,Y)) \leq \aleph_{0}$ for some compact  $G\subset Q_{P}(X,Y)$. Then $X$ is separable.
\end{theorem}

\begin{proof}
Let $g \in G$ such that $g(x) = b$, for all $x \in X$, where $b \in Y$ is fixed.  Let open set $U = [x_{1}, \dots, x_{n};V_1,\dots, V_{n}]$ in $Q_{P}(X,Y)$, where $x_{i} \in X $ and $V_{i}$ open in $Y$. Construct a set $K(U) = \{x_{1}, \dots, x_{n}\}$.
Since $\psi(G, Q_{P}(X,Y)) \leq \aleph_{0}$, then by definition of pseudocharacter there exists a collection $\mathcal{V} = \{O_{n} \colon n \in \mathbb{N}\}$ of open subsets of $Q_{P}(X,Y)$  such that $\bigcap \mathcal{V} = G$. So, for fixed $n \in \mathbb{N}$ and for each $g \in G$ we have an open set $U_{g}$ such that $g \in U_{g}\subset O_{n}$. Since $G$ is compact so every open cover $\{U_{g}\colon g \in G\}$ of $G$ has finite subcover say $\{U_{g_{1}}, \dots, U_{g_{m}}\}$. Construct two sets $P_{n} = \bigcup_{i = 1}^{m} U_{g_{i}}$ such that $G \subset P_{n} \subset O_{n}$ and $D_{n} = \bigcup_{i = 1}^{m}K(U_{g_{i}})$. It is easy to see that $D = \cup \{D_{n} \colon n \in \mathbb{N}\}$ is countable. Now we claim $\overline{D} = X$. Let us assume that $\overline{D} \neq X$, then $x \in X \setminus \overline{D}$. By the Lemma \eqref{l:Evalution map} the map $e_{x} \colon Q_{P}(X,Y) \to Y$ defined by $e_{x}(f) = f(x)$ is a continuous, therefore the set $e_{x}(G)$ is bounded in $Y$. So there exists $b' > b$ such that $|f(x)| < b'$ for all $f \in G$ and from Lemma \eqref{l:Quasicontinuous function} we have a  map $h \in Q_{P}(X,Y)$ with $h(x) = b'$ and $h(D)\subset \{b\}$. This implies $h(a) = b'$, so $h \notin G$. However $h|D = g|D$ implies $h \in \bigcap \mathcal{V}$, a contradiction.
\end{proof}

By the result \cite[Theorem 4.12]{Hola2020} that for a regular space $X$, network weight of $X$ follows the network weight of $Q_{P}(X)$ and in general network weight of any space $X$ dominates the pseudocharacter of the space $X$ by Problem \cite[Problem 156(iii)]{Tkachuk2010}. Now, in the next result we proved that the pseudocharacter of $Q_{P}(X,Y)$ lies between the network weight of $X$ and the network weight of $Q_{P}(X,Y)$.

\begin{theorem}\label{t:network to pesudo}
If $X$ is a regular space and $Y$ is any space, then $nw(X) \leq \psi(Q_{P}(X,Y))$.
\end{theorem}

\begin{proof}
Let $g(x) = z_{0}$ for all $x \in X$ and $\psi(g, Q_{P}(X,Y)) = |\mathcal{V}|$, where $\mathcal{V}$ is collection of open subsets of $Q_{P}(X,Y)$ such that  $\cap \mathcal{V} =\{g_{0}\}$ and $\mathcal{V}$ contains element of the form $U = [x_{1}, \dots, x_{n}; V_{1}, \dots, V_{n}]$. Now construct a set $K(U) = \{x_{1}, \dots, x_{n}\}$ and take $\mathcal{N} = \cup\{K(U)\colon U \in \mathcal{V} \}$, clearly $|\mathcal{N}| \leq |\mathcal{V}|$. Since  $\psi(g, Q_{P}(X,Y)) \leq \psi(Q_{P}(X,Y))$, so $|\mathcal{V}| \leq \psi(Q_{P}(X,Y))$. Now we claim that $\mathcal{N}$ is a network of $X$. Let us consider $V$ be any open subset of $X$ containing $x$. By Lemma \eqref{l:Condition for quasicont} define a quasi-continuous function as
 \[
f(x) =
\begin{cases}
z_{1}  & \text{if   $x \in V $};\\
z_{0} & \text{otherwise}.
\end{cases}
\]
There exists a $U \in \mathcal{V}$ such that $f \notin U$. Then there exists a $N \in \mathcal{N}$ such that $z_{0} \notin f(N)$. Therefore for each  $y \in N$ the $f(y)\neq z_{0}$. This implies  $y \notin V^{C}$, thus $y \in N \subset V$. Therefore $\mathcal{N}$ is network for $X$. Hence $nw(X)\leq \psi(Q_{P}(X,Y))$.
\end{proof}

\begin{corollary}\label{c:network relation}
For a regular space $X$ and  $Y$ be space, then $nw(X) \leq nw(Q_{P}(X,Y))$.
\end{corollary}

\begin{proof}
For a space $X$ the $\psi(X) \leq nw(X)$. Then by Theorem \eqref{t:network to pesudo}, we get $nw(X) \leq nw(Q_{P}(X,Y))$.
\end{proof}

\begin{corollary}
Let $X$ be a regular space and $Y$ be metric space then $d(X) \leq nw(Q_{P}(X,Y))$. \cite[Theorem 4.13]{Kumar2022}
\end{corollary}

\begin{proof}
For a space $X$ the $d(X) \leq nw(X)$. Then by Theorem \eqref{t:network to pesudo}, we get  $d(X) \leq nw(X) \leq nw(Q_{P}(X,Y))$.
\end{proof}

\begin{theorem}
For regular space $X$, then $wc(X) \cdot log(nw(X)) \leq iw(Q_{P}(X))$.
\end{theorem}

\begin{proof}
For any space $Z$, $\psi(Z)\cdot log(nw(Z))\leq iw(Z)$. From Theorem \eqref{t:weak covering pseudo character}, we have $wc(X) \leq \psi(Q_{P}(X))$ and by Corollary \eqref{c:network relation} $nw(X) \leq nw(Q_{P}(X))$. Therefore,  $wc(X) \cdot log(nw(X)) \leq iw(Q_{P}(X))$.
\end{proof}

In general $wc(X) \cdot log(nw(X)) \neq iw(Q_{P}(X))$. For example, from \cite[Example 5.1]{Hola2020} we have $\aleph_{1} = \psi(Q_{P}(\mathbb{R})) \leq iw(Q_{P}(\mathbb{R}))$, $wc(\mathbb{R}) = \aleph_{0}$ and $nw(\mathbb{R}) = \aleph_{0}$, so we have $log(nw(\mathbb{R})) = \aleph_{0}$. This implies $wc(\mathbb{R}) \cdot log(nw(\mathbb{R}) )\neq iw(Q_{P}(\mathbb{R}))$.

\begin{theorem}
For Hausdorff space $X$ and $Y$ be a space, then $d(F_{P}(X,Y)) \leq w(X) \cdot d(Y)$.
\end{theorem}

\begin{proof}
For any space $X$ and dense subset $S$ of $X$, $d(X) \leq d(S)$. By \cite[Theorem 4.15]{Kumar2022}, we have $Q_{P}(X,Y)$ which is dense in $F_{P}(X,Y)$. Therefore $d(F_{P}(X,Y)) \leq d(Q_{P}(X,Y))$. Also from \cite[Theorem 4.10]{Kumar2022}, $d(Q_{P}(X,Y)) \leq w(X) \cdot d(Y)$. Thus $d(F_{P}(X,Y)) \leq w(X) \cdot d(Y)$.
\end{proof}
\section{Mappings on the spaces $Q_{P}(X)$ and $Q_{P}(X,Y)$} \label{s:Mappings on the spaces qpx and qpxy}
Initially, in \cite[Chapter-II]{McCoy1988}, McCoy and Ntantu investigated some maps on the space $C_{P}(X,Y)$. Continuing this work, Kumar and Tyagi studied some maps on the space $Q_{P}(X,Y)$ in \cite{Kumar2022}. Now, in this section we study properties of maps on $Q_{P}(X)$ and $Q_{P}(X,Y)$ spaces.

\begin{theorem}
Let $X$ be a Hausdorff space and $Y$ be an open subspace of $X$. A restriction map $\pi_{Y} \colon Q_{P}(X) \to Q_{P}(Y)$ is defined by $\pi_{Y}(f) = f|Y$ for all $f \in Q_{P}(X)$. Then $\pi_{Y}$ is open continuous and $\pi_Y(Q_{P}(X)) = Q_{P}(Y)$.
\end{theorem}

\begin{proof}
The restriction map $\pi_{Y} \colon Q_{P}(X) \to Q_{P}(Y)$ is continuous since it is a projection map. For non-empty open subset $Y$ of $X$, we prove $\pi_{Y}(Q_{P}(X)) = Q_{P}(Y)$. It is oblivious $\pi_{Y}(Q_{P}(X)) \subset Q_{P}(Y)$. Since $Y$ is open subspace of $X$ and $g\in Q(Y)$, now we define a function $h:X\to \mathbb{R}$ as follows
\[h(x)=
\begin{cases}
g(x), & \text{if}\,  x \in Y; \\
1, &    X\setminus Y.
\end{cases}
\]
Clearly $\pi_{Y}(h) = h|Y = g$. Let any $x\in X$. Since $h|Y=g$ and $g$ is quasi-continuous. Therefore, for $x\in Y$ function $h$ is quasi-continuous.  If $x\in X\setminus Y$, then $h(x)=h(y)$ for all $y\in \overline{X\setminus Y}=X\setminus Y$ (Since $X\setminus Y$ is closed). Hence by Lemma \eqref{l:Condition for quasicont} $h$ is quasi-continuous at $x$. Finally we have $h$ is quasi-continuous. Next, First, we prove that $\pi_{Y}$ is open. Let $W(f, \{x_{1}, \dots, x_{k}\}, \epsilon)$ is open in $Q_{P}(X)$, where $k \in \mathbb{N}$. We claim $\pi_{Y}(W(f, \{x_{1}, \dots, x_{k}\}, \epsilon))$ is open in $Q_{P}(Y)$. For this we prove
\[\pi_{Y}(W(f, \{x_{1}, \dots, x_{k}\}, \epsilon)) = W(\pi_{Y}(f), \{x_{1}, \dots, x_{k}\}, \epsilon),\]
where $x_{i} \in Y$ for all $1 \leq i \leq k$. By definition of topology of point-wise convergence $W(\pi_{Y}(f), \{x_{1}, \dots, x_{k}\}, \epsilon)\subset Q_{P}(Y)$. By definition of map $\pi_{Y}$,
\[\pi_{Y}(W(f, \{x_{1}, \dots, x_{k}\}, \epsilon))\subset Q_{P}(Y)\]
Let $g \in W(f, \{x_{1}, \dots, x_{k}\}, \epsilon)$ then by definition of open set in $Q_{P}(X)$, $|g(x_{i}) - f(x_{i})| < \epsilon$ for all either $i = 1, 2, \dots, k$ or $i = 1, 2, \dots, l$, where $l \leq k$.

If $x \in Y$, implies $\pi_{Y}(g)(x) = g(x)$ and $\pi_{Y}(f)(x) = f(x)$. This show that $\pi_{Y}(g) \in  W(\pi_{Y}(f), \{x_{1}, \dots, x_{l}\}, \epsilon)$, therefore
\[\pi_{Y}(W(f, \{x_{1}, \dots, x_{k}\}, \epsilon))\subset W(\pi_{Y}(f), \{x_{1}, \dots, x_{l}\}, \epsilon)\]
Remain to prove $\pi_{Y}(W(\pi_{Y}(f), \{x_{1}, \dots, x_{k}\}, \epsilon))\supset W(\pi_{Y}(f), \{x_{1}, \dots, x_{l}\}, \epsilon)$. Take $g \in Q_{P}(Y)$ so there exists a $g_{1} \in Q_{P}(X)$ such that $g = \pi_{Y}(g_{1})$. Let
\[
m(x) =
\begin{cases}
g_{1}(x) & \text{if $x \notin {\bigcup_{i = l + 1}^{k}} \, \overline{V^{*}_{i}}$};\\
f(x) & \text{if $x \in {\bigcup_{i = l + 1}^{k}}\, \overline{V^{*}_{i}}$}.
 \end{cases}
\]
where $V^{*}_{i} = V_{i}\cap (X \backslash Y)$  and $V_{i}$ are disjoint open set containing point $x_{i}$, then  $m$ is quasi-continuous. Since $\pi_{Y}(m) = \pi_{Y}(g_{1}) = g$ implies that $m \in W(f, \{x_{1}, \dots, x_{k}\}, \epsilon)$ and $g\in \pi_{Y}(W(f, \{x_{1}, \dots, x_{k}\}, \epsilon))$. So,
\[\pi_{Y}(W(\pi_{Y}(f), \{x_{1}, \dots, x_{k}\}, \epsilon))\supset W(\pi_{Y}(f), \{x_{1}, \dots, x_{l}\}, \epsilon)\]
\end{proof}

As we know from the result \cite[Theorem, 1.2.4]{Neubrunn1988} that semi-continuity and quasi-continuity are equivalence for single-valued functions. Now from the result \cite[Remark, 13]{Levine1963} and  \eqref{t:product of cont and qcount is qcount} it is clear that the product of two quasi-continuous functions is not quasi-continuous in general but on the other hand the product of continuous and quasi-continuous function is quasi-continuous respectively. In general, the product $f \cdot g\colon X \to \mathbb{R}$ of two real-valued functions $f, g \colon X \to \mathbb{R}$ on the set $X$ is defined by $(f \cdot g)(x) = f(x) g(x)$.

\begin{theorem}\label{t:product of cont and qcount is qcount}
The product of continuous and quasi-continuous functions is quasi-continuous.
\end{theorem}

Now we are going to find out the condition under which the map defined on the product of $C_{P}(X)$ and $Q_{P}(X)$ into $Q_{P}(X)$ is continuous.

\begin{theorem}
The map $q \colon C_{P}(X) \times Q_{P}(X) \to Q_{P}(X)$ defined by $q(f, g) = f \cdot g$ for a space $X$ is a continuous map.
\end{theorem}

\begin{proof}
Let $h_{0} = (f_{0}, g_{0})$ be any point in $C_{P}(X) \times Q_{P}(X)$. Let us choose an open subset $V$ of $Q_{P}(X)$ that contain $q(h_{0}) = f_{0} \cdot g_{0}$. So, by the definition of topology of point-wise convergence there exists a subset $\{x_{1}, \dots, x_{n}\}$ of the space $X$ such that $q(h_{0})\in W(q(h_{0}),\{x_{1}, \dots, x_{n}\}, \epsilon)\subset V$, where $n \in \mathbb{N}$ and $\epsilon > 0$. Take
\[M = \sum_{i=1}^{n}|f_{0}(x_i)| + \sum_{i=1}^{n}|g_{0}(x_{i})| + 3 \,  \text{and}\, \delta = \min \left \{\dfrac{\epsilon}{2M},1 \right \}\]
Choose a open subset $O = O_{1} \times O_{2}$ of $C_{P}(X) \times Q_{P}(X)$ containing $h_0$, where $O_{1} = W(f_{0},\{x_{1}, \dots, x_{n}\}, \delta)$ and $O_{2} = W(g_{0},\{x_{1}, \dots, x_{n}\}, \delta)$ are open neighborhood of $f_{0}$ and $g_{0}$ in $C_{P}(X)$ and $Q_{P}(X)$ respectively. Therefore, for any $h = (f,g)\in O$ and for all $i \leq n$, where $n \in \mathbb{N}$ we have,
\[|g(x_{i})| < 1 + |g_{0}(x_{i})|< M \, \text{and} \, |f(x_{i})| < 1 + |f_{0}(x_{i})|< M \]
Now
\begin{align*}
|q(h)(x_{i}) - q(h_{0})(x_{i})| & = |(f\cdot g)(x_{i}) - (f_{0} \cdot g_{0})(x_{i})| = |f(x_{i}) \cdot g(x_{i}) - f_{0}(x_{i}) \cdot g_{0}(x_{i})| \\
& = |f(x_{i}) \cdot g(x_{i}) - f_{0}(x_{i}) \cdot g(x_{i}) + f_{0}(x_{i}) \cdot  g(x_{i}) - f_{0}(x_{i}) \cdot g_{0}(x_{i})| \\
& \leq|g(x_{i})||(f(x_{i}) - f_{0}(x_{i}))| + |f_{0}(x_{i})||(g(x_{i}) - g_{0}(x_{i}))|\\
&  < M \cdot \frac{\epsilon}{2M} + M \cdot \frac{\epsilon}{2M} < \epsilon, \, \text{for all}\, i \leq n \\
|q(h)(x_{i}) - q(h_{0})(x_{i})| & < \epsilon
\end{align*}
This shows that $q(h) = f \cdot g \in V$ and $q(O)\subset V$, therefore map $q$ is continuous.
\end{proof}

In \cite{Kumar2022}, Kumar and Tyagi studied the continuity of an induced map on $Q_{P}(X,Y)$. In continuation, we prove the denseness of the image of the space $Q_{P}(X,Y)$ under the induced induced map in the next result.

\begin{theorem}
Let $r \colon Y \to Z$ be a surjective continuous map and induced map $r_{\ast} \colon Q_{P}(X,Y) \to Q_{P}(X,Z)$ defined by $r_{\ast}(f) = r \circ f$. Then the space $r_{\ast}(Q_{P}(X,Y))$ is dense in the space $Q_{P}(X,Z)$.
\end{theorem}

\begin{proof}
To prove $r_{\ast}(Q_{P}(X,Y))$ is dense in $Q_{P}(X,Z)$ we show that for any $g \in Q_{P}(X,Z)$ there exists an open subset  $[x_{1}, \dots, x_{n}\colon V_{1}, \dots, V_{n}]$ of $Q_{P}(X,Z)$ containing $g$ such that $[x_{1}, \dots, x_{n}; V_{1}, \dots, V_{n}] \cap  r_{\ast}(Q_{P}(X,Y)) \neq \emptyset$, where $x_{i} \in X$ and $V_{i}$ open in the space $Z$. By definition of continuity and surjectivity of the map $r$, $r^{-1}(V_{i}) = U_{i}$ for some open set $U_{i}$ in $Y$, where $i \leq n$. Therefore, for any $f$ of $[x_{1}, \dots, x_{n}\colon U_{1}, \dots, U_{n}]\subset Q_{P}(X,Y)$, $r(f(x_{i})) = r_{\ast}(f)(x_{i})\in V_{i}$ for all $i\leq n$. So, $r_{\ast}(f) \in [x_{1}, \dots, x_{n}\colon V_{1}, \dots, V_{n}]$. Therefor, $[x_{1}, \dots, x_{n}; V_{1}, \dots, V_{n}] \cap  r_{\ast}(Q_{P}(X,Y)) \neq \emptyset$.
\end{proof}

\section{Conclusion}\label{sec13}
In this paper, we have studied cardinal invariants in the context of $Q_{P}(X,Y)$-spaces by examining the relationships among pseudocharacter, network weight, weight, and tightness. Our study demonstrates that the pseudocharacter of $Q_{P}(X,Y)$ dominates the network weight, density, and weak covering number of a regular space $X$, and we provided criteria $X$ is ordered Hausdorff space for that the weight and network weight of $Q_{P}(X)$ coincide. Also, we found that a regular space is regular whenever the  pseudocharacter of a compact subset of space $Q_{P}(X,Y)$ is countable. Furthermore, we analyzed the openness of the restriction map on $Q_{P}(X)$ and proved the image of $Q_{P}(X,Y)$ under the induced map is dense in $Q_{P}(X,Z)$. Our findings will enhance the study of cardinal invariants and properties of $Q_{P}(X,Y)$-spaces.

\bibliography{sn-bibliography}


\begin{thebibliography}{23}
\ifx \bisbn   \undefined \def \bisbn  #1{ISBN #1}\fi
\ifx \binits  \undefined \def \binits#1{#1}\fi
\ifx \bauthor  \undefined \def \bauthor#1{#1}\fi
\ifx \batitle  \undefined \def \batitle#1{#1}\fi
\ifx \bjtitle  \undefined \def \bjtitle#1{#1}\fi
\ifx \bvolume  \undefined \def \bvolume#1{\textbf{#1}}\fi
\ifx \byear  \undefined \def \byear#1{#1}\fi
\ifx \bissue  \undefined \def \bissue#1{#1}\fi
\ifx \bfpage  \undefined \def \bfpage#1{#1}\fi
\ifx \blpage  \undefined \def \blpage #1{#1}\fi
\ifx \burl  \undefined \def \burl#1{\textsf{#1}}\fi
\ifx \doiurl  \undefined \def \doiurl#1{\url{https://doi.org/#1}}\fi
\ifx \betal  \undefined \def \betal{\textit{et al.}}\fi
\ifx \binstitute  \undefined \def \binstitute#1{#1}\fi
\ifx \binstitutionaled  \undefined \def \binstitutionaled#1{#1}\fi
\ifx \bctitle  \undefined \def \bctitle#1{#1}\fi
\ifx \beditor  \undefined \def \beditor#1{#1}\fi
\ifx \bpublisher  \undefined \def \bpublisher#1{#1}\fi
\ifx \bbtitle  \undefined \def \bbtitle#1{#1}\fi
\ifx \bedition  \undefined \def \bedition#1{#1}\fi
\ifx \bseriesno  \undefined \def \bseriesno#1{#1}\fi
\ifx \blocation  \undefined \def \blocation#1{#1}\fi
\ifx \bsertitle  \undefined \def \bsertitle#1{#1}\fi
\ifx \bsnm \undefined \def \bsnm#1{#1}\fi
\ifx \bsuffix \undefined \def \bsuffix#1{#1}\fi
\ifx \bparticle \undefined \def \bparticle#1{#1}\fi
\ifx \barticle \undefined \def \barticle#1{#1}\fi
\bibcommenthead
\ifx \bconfdate \undefined \def \bconfdate #1{#1}\fi
\ifx \botherref \undefined \def \botherref #1{#1}\fi
\ifx \url \undefined \def \url#1{\textsf{#1}}\fi
\ifx \bchapter \undefined \def \bchapter#1{#1}\fi
\ifx \bbook \undefined \def \bbook#1{#1}\fi
\ifx \bcomment \undefined \def \bcomment#1{#1}\fi
\ifx \oauthor \undefined \def \oauthor#1{#1}\fi
\ifx \citeauthoryear \undefined \def \citeauthoryear#1{#1}\fi
\ifx \endbibitem  \undefined \def \endbibitem {}\fi
\ifx \bconflocation  \undefined \def \bconflocation#1{#1}\fi
\ifx \arxivurl  \undefined \def \arxivurl#1{\textsf{#1}}\fi
\csname PreBibitemsHook\endcsname

\bibitem[\protect\citeauthoryear{Baire}{1899}]{baire1899}
\begin{barticle}
\bauthor{\bsnm{Baire}, \binits{R.}}:
\batitle{Sur les functions des varaibles reells}.
\bjtitle{Ann.Mat. Pura Appl.}
\bvolume{3},
\bfpage{1}--\blpage{122}
(\byear{1899})
\end{barticle}
\endbibitem

\bibitem[\protect\citeauthoryear{Kempisty}{1932}]{Kempisty1932}
\begin{barticle}
\bauthor{\bsnm{Kempisty}, \binits{S.}}:
\batitle{Sur les fonctions quasicontinues}.
\bjtitle{Fundamenta Mathematicae}
\bvolume{1}(\bissue{19}),
\bfpage{184}--\blpage{197}
(\byear{1932})
\end{barticle}
\endbibitem

\bibitem[\protect\citeauthoryear{Piotrowski}{1985-86}]{Piotrowski1985}
\begin{barticle}
\bauthor{\bsnm{Piotrowski}, \binits{Z.}}:
\batitle{Separate and joint continuity}.
\bjtitle{Real Anal. Exchange}
\bvolume{11},
\bfpage{293}--\blpage{322}
(\byear{1985-86})
\end{barticle}
\endbibitem

\bibitem[\protect\citeauthoryear{Piotrowski}{1987}]{Piotrowski1987}
\begin{barticle}
\bauthor{\bsnm{Piotrowski}, \binits{Z.}}:
\batitle{A survey of results concerning generalized continuity in topological spaces}.
\bjtitle{Acta Math. Univ. Comen.}
\bvolume{52-53},
\bfpage{91}--\blpage{110}
(\byear{1987})
\end{barticle}
\endbibitem

\bibitem[\protect\citeauthoryear{Neubrunn}{1988}]{Neubrunn1988}
\begin{barticle}
\bauthor{\bsnm{Neubrunn}, \binits{T.}}:
\batitle{Quasi-continuity}.
\bjtitle{Real Anal. Exchange}
\bvolume{14}(\bissue{2}),
\bfpage{259}--\blpage{306}
(\byear{1988})
\end{barticle}
\endbibitem

\bibitem[\protect\citeauthoryear{Bishnoi and Mishra}{2023}]{Bishnoi2023}
\begin{barticle}
\bauthor{\bsnm{Bishnoi}, \binits{C.M.}},
\bauthor{\bsnm{Mishra}, \binits{S.}}:
\batitle{Quasicontinuous function on strong forms of connected space}.
\bjtitle{J.Indones.Math.Soc.}
\bvolume{29}(\bissue{1}),
\bfpage{106}--\blpage{115}
(\byear{2023})
\end{barticle}
\endbibitem

\bibitem[\protect\citeauthoryear{Hola et~al.}{2021}]{Hola2021}
\begin{bbook}
\bauthor{\bsnm{Hola}, \binits{L.}},
\bauthor{\bsnm{Holy}, \binits{D.}},
\bauthor{\bsnm{Moors}, \binits{W.}}:
\bbtitle{USCO and Quasicontinuous Mappings}
vol. \bseriesno{81}.
\bpublisher{De Gruyter},
\blocation{Berlin}
(\byear{2021})
\end{bbook}
\endbibitem

\bibitem[\protect\citeauthoryear{Arkhangel'skii}{1992}]{Arkhangelskii992}
\begin{bbook}
\bauthor{\bsnm{Arkhangel'skii}, \binits{A.V.}}:
\bbtitle{Topological Function Spaces}
vol. \bseriesno{78}.
\bpublisher{Kluwer Academic Publishers},
\blocation{Boston}
(\byear{1992})
\end{bbook}
\endbibitem

\bibitem[\protect\citeauthoryear{Tkachuk}{2010}]{Tkachuk2010}
\begin{bbook}
\bauthor{\bsnm{Tkachuk}, \binits{V.V.}}:
\bbtitle{A $C_{p}$-theory Problem book:Topological and Function Spaces}.
\bpublisher{Springer},
\blocation{New {Y}ork}
(\byear{2010})
\end{bbook}
\endbibitem

\bibitem[\protect\citeauthoryear{Tkachuk}{2014}]{Tkachuk2014}
\begin{bbook}
\bauthor{\bsnm{Tkachuk}, \binits{V.V.}}:
\bbtitle{A $C_{p}$-theory Problem book:Special Features of Function Spaces}.
\bpublisher{Springer},
\blocation{New {Y}ork}
(\byear{2014})
\end{bbook}
\endbibitem

\bibitem[\protect\citeauthoryear{Tkachuk}{2015}]{Tkachuk2015}
\begin{bbook}
\bauthor{\bsnm{Tkachuk}, \binits{V.V.}}:
\bbtitle{A $C_{p}$-theory Problem book:Compactness in Function Spaces}.
\bpublisher{Springer},
\blocation{New {Y}ork}
(\byear{2015})
\end{bbook}
\endbibitem

\bibitem[\protect\citeauthoryear{Tkachuk}{2016}]{Tkachuk2016}
\begin{bbook}
\bauthor{\bsnm{Tkachuk}, \binits{V.V.}}:
\bbtitle{A $C_{p}$-theory Problem book:Functional Equivalencies}.
\bpublisher{Springer},
\blocation{New {Y}ork}
(\byear{2016})
\end{bbook}
\endbibitem

\bibitem[\protect\citeauthoryear{McCoy et~al.}{2018}]{McCoy2018}
\begin{bbook}
\bauthor{\bsnm{McCoy}, \binits{R.A.}},
\bauthor{\bsnm{Kundu}, \binits{S.}},
\bauthor{\bsnm{Jindal}, \binits{V.}}:
\bbtitle{Function Spaces with Uniform, Fine and Graph Topologies}.
\bpublisher{Springer},
\blocation{Berlin}
(\byear{2018})
\end{bbook}
\endbibitem

\bibitem[\protect\citeauthoryear{Mishra and Bhaumik}{2023}]{Mishra2023}
\begin{barticle}
\bauthor{\bsnm{Mishra}, \binits{S.}},
\bauthor{\bsnm{Bhaumik}, \binits{A.}}:
\batitle{Properties of function space under cauchy convergence topology}.
\bjtitle{Topol.Appl.}
\bvolume{338},
\bfpage{108653}
(\byear{2023})
\end{barticle}
\endbibitem

\bibitem[\protect\citeauthoryear{Aaliya and Mishra}{2023}]{Aaliya2023}
\begin{barticle}
\bauthor{\bsnm{Aaliya}, \binits{M.}},
\bauthor{\bsnm{Mishra}, \binits{S.}}:
\batitle{Space of homeomorphisms under regular topology}.
\bjtitle{Commun.Korean Math.Soc.}
\bvolume{38},
\bfpage{1299}--\blpage{1307}
(\byear{2023})
\end{barticle}
\endbibitem

\bibitem[\protect\citeauthoryear{Hola and Holy}{2016}]{Hola2016}
\begin{barticle}
\bauthor{\bsnm{Hola}, \binits{L.}},
\bauthor{\bsnm{Holy}, \binits{D.}}:
\batitle{Quasicontinuous subcontinuous functions and compactness}.
\bjtitle{Mediterr.J.Math.}
\bvolume{13}(\bissue{6}),
\bfpage{4509}--\blpage{4518}
(\byear{2016})
\end{barticle}
\endbibitem

\bibitem[\protect\citeauthoryear{Hola and Holy}{2017}]{Hola2017}
\begin{barticle}
\bauthor{\bsnm{Hola}, \binits{L.}},
\bauthor{\bsnm{Holy}, \binits{D.}}:
\batitle{Quasicontinuous functions and compactness}.
\bjtitle{Mediterr.J.Math.}
\bvolume{14}(\bissue{6}),
\bfpage{1}--\blpage{11}
(\byear{2017})
\end{barticle}
\endbibitem

\bibitem[\protect\citeauthoryear{Hola and Holy}{2018}]{Hola2018}
\begin{barticle}
\bauthor{\bsnm{Hola}, \binits{L.}},
\bauthor{\bsnm{Holy}, \binits{D.}}:
\batitle{Metrizability of the space of quasicontinuous functions}.
\bjtitle{Topol.Appl.}
\bvolume{246},
\bfpage{137}--\blpage{143}
(\byear{2018})
\end{barticle}
\endbibitem

\bibitem[\protect\citeauthoryear{Kumar and Tyagi}{2022}]{Kumar2022}
\begin{barticle}
\bauthor{\bsnm{Kumar}, \binits{M.}},
\bauthor{\bsnm{Tyagi}, \binits{B.K.}}:
\batitle{Cardinal invariants and special maps of quasicontinuous functions with the topology of pointwise convergence}.
\bjtitle{Appl. Gen. Topol.}
\bvolume{23}(\bissue{2}),
\bfpage{303}--\blpage{314}
(\byear{2022})
\end{barticle}
\endbibitem

\bibitem[\protect\citeauthoryear{Kunen and Vaughan}{1984}]{Kunen2014}
\begin{bbook}
\bauthor{\bsnm{Kunen}, \binits{K.}},
\bauthor{\bsnm{Vaughan}, \binits{J.K.}}:
\bbtitle{Handbook of Set-theoretic Topology}.
\bpublisher{Elsevier Science},
\blocation{Amsterdam}
(\byear{1984})
\end{bbook}
\endbibitem

\bibitem[\protect\citeauthoryear{Hola and Holy}{2020}]{Hola2020}
\begin{barticle}
\bauthor{\bsnm{Hola}, \binits{L.}},
\bauthor{\bsnm{Holy}, \binits{D.}}:
\batitle{Quasicontinuous functions and the topology of pointwise convergence}.
\bjtitle{Topol.Appl.}
\bvolume{282},
\bfpage{107301}
(\byear{2020})
\end{barticle}
\endbibitem

\bibitem[\protect\citeauthoryear{McCoy and Ntantu}{1988}]{McCoy1988}
\begin{bbook}
\bauthor{\bsnm{McCoy}, \binits{R.A.}},
\bauthor{\bsnm{Ntantu}, \binits{I.}}:
\bbtitle{Topological Properties of Spaces of Continuous Functions, Lecture Notes in Mathematics}
vol. \bseriesno{1315}.
\bpublisher{Springer},
\blocation{Berlin}
(\byear{1988})
\end{bbook}
\endbibitem

\bibitem[\protect\citeauthoryear{N.Levine}{1963}]{Levine1963}
\begin{barticle}
\bauthor{\bsnm{N.Levine}}:
\batitle{Semi-open sets and semi-continuity in topological spaces}.
\bjtitle{Amer. Math. Monthly}
\bvolume{70}(\bissue{1}),
\bfpage{36}--\blpage{41}
(\byear{1963})
\end{barticle}
\endbibitem

\end{thebibliography}

\end{document}